\documentclass[12pt]{article}

\makeatletter
\newcommand{\doublewidetilde}[1]{{%
  \mathpalette\double@widetilde{#1}%
}}
\newcommand{\double@widetilde}[2]{%
  \sbox\z@{$\m@th#1\widetilde{#2}$}%
  \ht\z@=.9\ht\z@
  \widetilde{\box\z@}%
}
\usepackage{xpatch}
\usepackage[fleqn]{amsmath}
\usepackage{mathtools}
\usepackage{comment}
\DeclarePairedDelimiter\ord{\lvert}{\rvert}
\providecommand\divides{} 
\renewcommand*\divides{\mathrel{|}} 
\makeatother

\usepackage{amsmath,tabu}
\usepackage{tikz-cd}
\usepackage[utf8]{inputenc}
\usepackage{color}
\usepackage{ctable}
\usepackage{array}
\newcolumntype{P}[1]{>{\centering\arraybackslash}p{#1}}
\usepackage{enumerate}
\usepackage{amsthm}
\usepackage{amsfonts}
\usepackage{color}
\usepackage{multicol}
\usepackage{array}
\newcolumntype{L}{>{$}l<{$}}
\usepackage{hyperref}
\usepackage[margin=1in]{geometry}
\usepackage{enumitem}  
\begingroup
\catcode`\&=13
\gdef\pampmatrix{%
  \begingroup
  \let&=\amsamp
  \begin{pmatrix}%
}
\gdef\endpampmatrix{\end{pmatrix}\endgroup}
\endgroup
\usepackage{stackengine}
\newtheorem{theorem}{Theorem}[section]

\newtheorem{lemma}[theorem]{Lemma}

\newtheorem{remark}[theorem]{Remark}
\stackMath
\usepackage{amssymb}
\usepackage{graphicx}
\usepackage{hyperref}
\newcommand{\octada}[8]{
\begin{tikzpicture}
\draw (0,1.05) -- (1.8,1.05) -- (1.8,-1.05) -- (0,-1.05) -- (0,1.05);
\draw[anchor=east] (0.8,0.75) node {$#1$};
\draw[anchor=east] (1.6,0.75) node {$#5$};
\draw[anchor=east] (0.8,0.25) node {$#2$};
\draw[anchor=east] (1.6,0.25) node {$#6$};
\draw[anchor=east] (0.8,-0.25) node {$#3$};
\draw[anchor=east] (1.6,-0.25) node {$#7$};
\draw[anchor=east] (0.8,-0.75) node {$#4$};
\draw[anchor=east] (1.6,-0.75) node {$#8$};
\end{tikzpicture}
}

\newcommand{\octadb}[8]{
\begin{tikzpicture}
\draw (0,1.05) -- (1.8,1.05) -- (1.8,-1.05) -- (0,-1.05);
\draw[anchor=east] (0.8,0.75) node {$#1$};
\draw[anchor=east] (1.6,0.75) node {$#5$};
\draw[anchor=east] (0.8,0.25) node {$#2$};
\draw[anchor=east] (1.6,0.25) node {$#6$};
\draw[anchor=east] (0.8,-0.25) node {$#3$};
\draw[anchor=east] (1.6,-0.25) node {$#7$};
\draw[anchor=east] (0.8,-0.75) node {$#4$};
\draw[anchor=east] (1.6,-0.75) node {$#8$};
\end{tikzpicture}
}

\newcommand{\octadba}[7]{
\begin{tikzpicture}
\draw (0,1.05) -- (1.8,1.05) -- (1.8,-1.05) -- (0,-1.05);
\draw[anchor=east] (0.8,0.75) node {$#1$};
\draw[anchor=east] (1.6,0.75) node {$#5$};
\draw[anchor=east] (1.4,0.25) node {$#2$};
\draw[anchor=east] (0.8,-0.25) node {$#3$};
\draw[anchor=east] (1.6,-0.25) node {$#6$};
\draw[anchor=east] (0.8,-0.75) node {$#4$};
\draw[anchor=east] (1.6,-0.75) node {$#7$};
\end{tikzpicture}
}

\newcommand{\octadbb}[7]{
\begin{tikzpicture}
\draw (0,1.05) -- (1.8,1.05) -- (1.8,-1.05) -- (0,-1.05);
\draw[anchor=east] (0.8,0.75) node {$#1$};
\draw[anchor=east] (1.6,0.75) node {$#5$};
\draw[anchor=east] (0.8,0.25) node {$#2$};
\draw[anchor=east] (1.6,0.25) node {$#6$};
\draw (1.0,-0.25) node {$#3$};
\draw[anchor=east] (0.8,-0.75) node {$#4$};
\draw[anchor=east] (1.6,-0.75) node {$#7$};
\end{tikzpicture}
}

\newcommand{\octadbc}[6]{
\begin{tikzpicture}
\draw (0,1.05) -- (1.8,1.05) -- (1.8,-1.05) -- (0,-1.05);
\draw[anchor=east] (0.8,0.75) node {$#1$};
\draw[anchor=east] (1.6,0.75) node {$#5$};
\draw[anchor=east] (0.8,0.25) node {$#2$};
\draw[anchor=east] (1.6,0.25) node {$#6$};
\draw (1.0,-0.20) node {$#3$};
\draw (1.0,-0.70) node {$#4$};
\end{tikzpicture}
}

\newcommand{\octadc}[9]{
\begin{tikzpicture}
\draw (0,1.05) -- (1.8,1.05) -- (1.8,-1.05) -- (0,-1.05);
\draw[anchor=east] (0.8,0.75) node {$#1$};
\draw[anchor=east] (1.6,0.75) node {$#5$};
\draw[anchor=east] (0.8,0.25) node {$#2$};
\draw[anchor=east] (1.6,0.25) node {$#6$};
\draw[anchor=east] (0.8,-0.25) node {$#3$};
\draw[anchor=east] (1.6,-0.25) node {$#7$};
\draw[anchor=east] (0.8,-0.75) node {$#4$};
\draw[anchor=east] (1.6,-0.75) node {$#8$};
\draw (2.0,1.05) node {$#9$};
\end{tikzpicture}
}

\newcommand{\hoctad}[9]{
\begin{tabular}{|p{.13in}p{.13in}|p{.13in}p{.13in}|}
\hline
\makebox[0.13in][r]{#1}&\makebox[0.13in][r]{#3}&\makebox[0.13in][r]{#5}&\makebox[0.13in][r]{#7}\\
\makebox[0.13in][r]{#2}&\makebox[0.13in][r]{#4}&\makebox[0.13in][r]{#6}&\makebox[0.13in][r]{#8}\\
\hline
\end{tabular}^{#9}
}

\newcommand{\singlex}[8]{
\begin{tabular}{|p{.13in}p{.13in}|}
\hline
\makebox[0.13in][r]{$#1$}&\makebox[0.13in][r]{$#5$}\\
\makebox[0.13in][r]{$#2$}&\makebox[0.13in][r]{$#6$}\\
\makebox[0.13in][r]{$#3$}&\makebox[0.13in][r]{$#7$}\\
\makebox[0.13in][r]{$#4$}&\makebox[0.13in][r]{$#8$}\\
\hline
\end{tabular}^{\times}
}

\newcommand{\singlenox}[8]{
\begin{tabular}{|p{.13in}p{.13in}|}
\hline
\makebox[0.13in][r]{$#1$}&\makebox[0.13in][r]{$#5$}\\
\makebox[0.13in][r]{$#2$}&\makebox[0.13in][r]{$#6$}\\
\makebox[0.13in][r]{$#3$}&\makebox[0.13in][r]{$#7$}\\
\makebox[0.13in][r]{$#4$}&\makebox[0.13in][r]{$#8$}\\
\hline
\end{tabular}
}

\newcommand{\smalloctada}[8]{
\begin{tikzpicture}[scale=0.84]
\draw (0,1.05) -- (1.18,1.05) -- (1.18,-1.05) -- (0,-1.05) -- (0,1.05);
\draw[anchor=east] (0.68,0.75) node {$\scriptstyle{#1}$};
\draw[anchor=east] (1.16,0.75) node {$\scriptstyle{#5}$};
\draw[anchor=east] (0.68,0.25) node {$\scriptstyle{#2}$};
\draw[anchor=east] (1.16,0.25) node {$\scriptstyle{#6}$};
\draw[anchor=east] (0.68,-0.25) node {$\scriptstyle{#3}$};
\draw[anchor=east] (1.16,-0.25) node {$\scriptstyle{#7}$};
\draw[anchor=east] (0.68,-0.75) node {$\scriptstyle{#4}$};
\draw[anchor=east] (1.16,-0.75) node {$\scriptstyle{#8}$};
\end{tikzpicture}
}

\newcommand{\smalloctadb}[8]{
\begin{tikzpicture}[scale=0.84]
\draw (0,1.05) -- (1.18,1.05) -- (1.18,-1.05) -- (0,-1.05);
\draw[anchor=east] (0.68,0.75) node {$\scriptstyle{#1}$};
\draw[anchor=east] (1.16,0.75) node {$\scriptstyle{#5}$};
\draw[anchor=east] (0.68,0.25) node {$\scriptstyle{#2}$};
\draw[anchor=east] (1.16,0.25) node {$\scriptstyle{#6}$};
\draw[anchor=east] (0.68,-0.25) node {$\scriptstyle{#3}$};
\draw[anchor=east] (1.16,-0.25) node {$\scriptstyle{#7}$};
\draw[anchor=east] (0.68,-0.75) node {$\scriptstyle{#4}$};
\draw[anchor=east] (1.16,-0.75) node {$\scriptstyle{#8}$};
\end{tikzpicture}
}

\newcommand{\smalloctadba}[7]{
\begin{tikzpicture}[scale=0.84]
\draw (0,1.05) -- (1.18,1.05) -- (1.18,-1.05) -- (0,-1.05);
\draw[anchor=east] (0.68,0.75) node {$\scriptstyle{#1}$};
\draw[anchor=east] (1.16,0.75) node {$\scriptstyle{#5}$};
\draw[anchor=east] (1.055,0.25) node {$\scriptstyle{#2}$};
\draw[anchor=east] (0.68,-0.25) node {$\scriptstyle{#3}$};
\draw[anchor=east] (1.16,-0.25) node {$\scriptstyle{#6}$};
\draw[anchor=east] (0.68,-0.75) node {$\scriptstyle{#4}$};
\draw[anchor=east] (1.16,-0.75) node {$\scriptstyle{#7}$};
\end{tikzpicture}
}

\newcommand{\smalloctadc}[9]{
\begin{tikzpicture}[scale=0.84]
\draw (0,1.05) -- (1.18,1.05) -- (1.18,-1.05) -- (0,-1.05);
\draw[anchor=east] (0.68,0.75) node {$\scriptstyle{#1}$};
\draw[anchor=east] (1.16,0.75) node {$\scriptstyle{#5}$};
\draw[anchor=east] (0.68,0.25) node {$\scriptstyle{#2}$};
\draw[anchor=east] (1.16,0.25) node {$\scriptstyle{#6}$};
\draw[anchor=east] (0.68,-0.25) node {$\scriptstyle{#3}$};
\draw[anchor=east] (1.16,-0.25) node {$\scriptstyle{#7}$};
\draw[anchor=east] (0.68,-0.75) node {$\scriptstyle{#4}$};
\draw[anchor=east] (1.16,-0.75) node {$\scriptstyle{#8}$};
\draw (1.38,1.05) node {$#9$};
\end{tikzpicture}
}

\newcommand{\octadempty}[1]{
\begin{tikzpicture}
\draw (0,1.05) node {};
\draw (0,0) node {$#1$};
\draw (0,-1.05) node {};
\end{tikzpicture}
}

\newcommand{\octadaline}[8]{
\begin{tikzpicture}
\draw (0,1.1) -- (1.6,1.1) -- (1.6,-1.1) -- (0,-1.1) -- (0,1.1);
\draw (0.4,0.8) node {$#1$};
\draw (1.2,0.8) node {$#5$};
\draw (0.4,0.3) node {$#2$};
\draw (1.2,0.3) node {$#6$};
\draw (0.4,-0.3) node {$#3$};
\draw (1.2,-0.3) node {$#7$};
\draw (0.4,-0.8) node {$#4$};
\draw (1.2,-0.8) node {$#8$};
\end{tikzpicture}
}

\newcommand{\octadbline}[8]{
\begin{tikzpicture}
\draw (0,1.1) -- (1.6,1.1) -- (1.6,-1.1) -- (0,-1.1);
\draw (0.4,0.8) node {$#1$};
\draw (1.2,0.8) node {$#5$};
\draw (0.4,0.3) node {$#2$};
\draw (1.2,0.3) node {$#6$};
\draw (0.4,-0.3) node {$#3$};
\draw (1.2,-0.3) node {$#7$};
\draw (0.4,-0.8) node {$#4$};
\draw (1.2,-0.8) node {$#8$};
\end{tikzpicture}
}

\newcommand{\octadcline}[9]{
\begin{tikzpicture}
\draw (0,1.1) -- (1.6,1.1) -- (1.6,-1.1) -- (0,-1.1);
\draw (0,0) -- (1.6,0);
\draw (0.4,0.8) node {$#1$};
\draw (1.2,0.8) node {$#5$};
\draw (0.4,0.3) node {$#2$};
\draw (1.2,0.3) node {$#6$};
\draw (0.4,-0.3) node {$#3$};
\draw (1.2,-0.3) node {$#7$};
\draw (0.4,-0.8) node {$#4$};
\draw (1.2,-0.8) node {$#8$};
\draw (1.8,1.1) node {$#9$};
\end{tikzpicture}
}
\newcommand{\octademptyline}[1]{
\begin{tikzpicture}
\draw (0,1.1) node {};
\draw (0,0) node {$#1$};
\draw (0,-1.1) node {};
\end{tikzpicture}
}

\newcommand{\octadex}[9]{
\begin{tabular}{|p{.05in}p{.05in}|p{.05in}p{.05in}|p{.05in}p{.05in}|}
\hline
\makebox[0.08in][r]{$\pm4$}&\makebox[0.08in][r]{#1}&\makebox[0.08in][r]{$#2$}&\makebox[0.08in][r]{$#6$}&\makebox[0.08in][r]{}&\makebox[0.08in][r]{}\\
\makebox[0.08in][r]{}&\makebox[0.08in][r]{}&\makebox[0.08in][r]{$#3$}&\makebox[0.08in][r]{$#7$}&\makebox[0.08in][r]{}&\makebox[0.08in][r]{}\\
\makebox[0.08in][r]{}&\makebox[0.08in][r]{}&\makebox[0.08in][r]{$#4$}&\makebox[0.08in][r]{$#8$}&\makebox[0.08in][r]{}&\makebox[0.08in][r]{}\\
\makebox[0.08in][r]{}&\makebox[0.08in][r]{}&\makebox[0.08in][r]{$#5$}&\makebox[0.08in][r]{$#9$}&\makebox[0.08in][r]{}&\makebox[0.08in][r]{}\\
\hline
\end{tabular}^{\times}
}

\tikzset{
  on each segment/.style={
    decorate,
    decoration={
      show path construction,
      moveto code={},
      lineto code={
        \path [#1]
        (\tikzinputsegmentfirst) -- (\tikzinputsegmentlast);
      },
      curveto code={
        \path [#1] (\tikzinputsegmentfirst)
        .. controls
        (\tikzinputsegmentsupporta) and (\tikzinputsegmentsupportb)
        ..
        (\tikzinputsegmentlast);
      },
      closepath code={
        \path [#1]
        (\tikzinputsegmentfirst) -- (\tikzinputsegmentlast);
      },
    },
  },
  mid arrow/.style={postaction={decorate,decoration={
        markings,
        mark=at position .5 with {\arrow[#1]{stealth}}
      }}},
}

\newcommand{\epssub}{\varepsilon_{\begin{tikzpicture}
\draw (0,0.45) -- (1.5,0.45) -- (1.5,-0.45) -- (0,-0.45) -- (0,0.45);
\draw (0.5,-0.45) -- (0.5,0.45);
\draw (1.0,-0.45) -- (1.0,0.45);
\path (0.1,0.3) node[draw,shape=circle,fill=black,scale=0.075] (p1) {};
\path (0.4,0.1) node[draw,shape=circle,fill=black,scale=0.075] (p6) {};
\path (0.6,0.1) node[draw,shape=circle,fill=black,scale=0.075] (p10) {};
\path (0.9,0.1) node[draw,shape=circle,fill=black,scale=0.075] (p14) {};
\path (0.9,-0.1) node[draw,shape=circle,fill=black,scale=0.075] (p15) {};
\path (0.9,-0.3) node[draw,shape=circle,fill=black,scale=0.075] (p16) {};
\path (1.1,-0.1) node[draw,shape=circle,fill=black,scale=0.075] (p19) {};
\path (1.4,-0.3) node[draw,shape=circle,fill=black,scale=0.075] (p24) {};
\end{tikzpicture}
}}
\newcommand{\Stab}{\mathrm{Stab}}
\newcommand{\Dih}{\mathrm{Dih}}
\newcommand{\Sym}{\mathrm{Sym}}
\newcommand{\Alt}{\mathrm{Alt}}
\newcommand{\Fix}{\mathrm{Fix}}
\newcommand{\M}{\mathrm{M}}
\newcommand{\PSL}{\mathrm{PSL}}
\begin{document}
\title{\Huge{ A Note on the Rank 5 Polytopes of $\M_{24}$}}
\author{Veronica Kelsey, Robert Nicolaides, Peter Rowley \footnote{Email address of corresponding author: peter.j.rowley@manchester.ac.uk  \newline Key words: abstract regular polytope, maximal rank, Mathieu group, MOG \newline MSC: 52B15, 52B05, 20D08} \\
 \small{Department of Mathematics, University of Manchester, Oxford Road, M13 6PL, UK}}
\maketitle
\begin{abstract}
The maximal rank of an abstract regular polytope for $\M_{24}$, the Mathieu group of degree 24, is 5. There are four such polytopes of rank 5 and in this note we describe them using Curtis's MOG. This description is then used to give an upper bound for the diameter of the chamber graphs of these polytopes.
\end{abstract}
\date{}
\maketitle

\section{Introduction}

Investigations into abstract regular polytopes for the Mathieu group $\M_{24}$ by Hartley and Hulpke \cite{hartleyhulpke} revealed that, in total, there are 1,260 such polytopes. Of these there are only four of rank 5, the highest rank of the $\M_{24}$ polytopes, with these forming two dual pairs. Polytopes of maximal rank, at least $5$, for a particular group are usually relatively few in number and therefore deserve further attention. See \cite{cameronfernandesleemnas}, \cite{fernandesleemans1} and \cite{fernandesleemans2} for examples of work in this direction. Among the sporadic finite simple simple groups $\M_{24}$ is pre-eminent because of both the richness of the many
 associated combinatorial objects (see, for example, the \textsc{Atlas} \cite{atlas} and \cite{conwaysloane}) and its influence on many of the other sporadic groups. The purpose of this short note is to display the rank 5 abstract regular polytopes of $\M_{24}$, using Curtis's (amazing) MOG \cite{curtis1} as the backdrop, with the aim of making these polytopes more transparent. As an application of these descriptions we probe the chamber graphs of these rank 5 polytopes, proving the following theorem.

\begin{theorem}\label{maintheorem} Suppose $\mathcal{P}$ is a rank 5 abstract regular polytope for $\M_{24}$. 
\begin{enumerate}
\item [(i)] If $\mathcal{P}$ has Schlafli symbol [4,10,3,4], then its chamber graph has diameter at most 127.
\item[(ii)] If $\mathcal{P}$ has Schlafli symbol [4,10,3,3], then its chamber graph has diameter at most 133.
\end{enumerate}
\end{theorem}

The number of chambers of our polytopes is the order of $\M_{24}$ which is 244,823,040. Much of the work so far on analysing polytopes often makes considerable use of machine calculations. Here, in the proof of Theorem~\ref{maintheorem}, we make limited use of machine calculations -- specifically to investigate the group $\PSL_2(11):2$ of order 1320. In fact, with patience, this could be done without reliance on machine.

\section{The Rank 5 Polytopes}

As is well-known, abstract regular polytopes and C-strings are equivalent formulations \cite{mcmullenschulte}. Here we approach the rank 5 polytopes of $\M_{24}$ via the route of C-strings. We recall that a set of involutions $\{t_1, \dots t_n\} $ is a C-string for a group $G$ if they generate $G$ and, setting $I =\{1, \dots, n \}$, satisfy
\begin{enumerate}
\item [(i)] for $i, j \in I$ with $|i - j| \geq 2, t_it_j = t_jt_i$; and
\item [(ii)] for $J, K \subseteq I, G_J \cap G_K = G_{J \cap K}$.
\end{enumerate} 
Here, for $\emptyset \neq J \subseteq I$, we put $G_J = \langle t_i \; | \; i \in J \rangle$ and $G_{\emptyset} =1$. In the context of C-strings, the chamber graph of the corresponding abstract regular polytope, is just the Cayley graph with respect to the generating set $\{t_1, \dots t_n\} $. In this graph, for $g \in G$ and $i \in \mathbb{N}$, denote the vertices of the Cayley graph distance $i$ from $g$ by $\Delta_i(g)$.

We now let $\Omega$ be a $24$-element set, equipped with the Steiner system $S(5,8,24)$ as provided by Curtis's MOG \cite{curtis1}, and let $G$ denote the automorphism group of this Steiner system. Then $G \cong \M_{24}$. Sometimes we shall use Curtis's labelling of the elements of $\Omega$ given in \cite{curtis1}. We shall encounter the following subset of $\Omega$
$$D =  \begin{array}{c}
\smalloctada{}{\times}{\times}{}{}{\times}{}{\times}
\smalloctadb{}{\times}{}{\times}{}{\times}{\times}{}
\smalloctadc{}{\times}{}{\times}{}{\times}{\times}{}{}
\end{array}.$$
Since $D$ is the symmetric difference of the two octads
$$ \begin{array}{c}
\smalloctada{}{\times}{\times}{\times}{\times}{}{}{}
\smalloctadb{}{\times}{}{}{}{\times}{}{}
\smalloctadc{}{\times}{}{}{}{\times}{}{}{}
\end{array} \text{ and } \begin{array}{c}
\smalloctada{}{}{}{\times}{\times}{\times}{}{\times}
\smalloctadb{}{}{}{\times}{}{}{\times}{}
\smalloctadc{}{}{}{\times}{}{}{\times}{}{}
\end{array},$$
it is a dodecad of $\Omega$. Hence the partition $\{ D, \Omega \setminus D \}$ is a duum. 
Put $L = \Stab_G(\{ D, \Omega \setminus D \})$. 
Then $L \cong \M_{12}:2$ (see \cite{curtis1}).

We define the following involutions of $G$ -- that they belong to $G$ may be confirmed either from \cite{curtis1} or using \cite{curtis2}. Below a pair of elements of $\Omega$ in the MOG diagram joined by a line means the involution interchanges those two elements, while a single dot indicates that the involution fixes that element of $\Omega$.
\begin{multicols}{3}
$$\begin{tikzpicture}
 \path (0,1.05) node[draw=none] (p0) {\phantom{$g_3=$}};
 \path (0,0.0) node[anchor=north] (p1) {$g_1=$};
 \path (0,-1.05) node[draw=none] (p0) {\phantom{$g_3=$}};
\end{tikzpicture}
\begin{tikzpicture}
\draw (0,0.9) -- (3.0,0.9) -- (3.0,-0.9) -- (0,-0.9) -- (0,0.9);
\draw (1.0,-0.9) -- (1.0,0.9);
\draw (2.0,-0.9) -- (2.0,0.9);
\path (0.2,0.6) node[draw,shape=circle,fill=black,scale=0.15] (p1) {};
\path (0.2,0.2) node[draw,shape=circle,fill=black,scale=0.15] (p2) {};
\path (0.2,-0.2) node[draw,shape=circle,fill=black,scale=0.15] (p3) {};
\path (0.2,-0.6) node[draw,shape=circle,fill=black,scale=0.15] (p4) {};
\path (0.8,0.6) node[draw,shape=circle,fill=black,scale=0.15] (p5) {};
\path (0.8,0.2) node[draw,shape=circle,fill=black,scale=0.15] (p6) {};
\path (0.8,-0.2) node[draw,shape=circle,fill=black,scale=0.15] (p7) {};
\path (0.8,-0.6) node[draw,shape=circle,fill=black,scale=0.15] (p8) {};
\path (1.2,0.6) node[draw,shape=circle,fill=black,scale=0.15] (p9) {};
\path (1.2,0.2) node[draw,shape=circle,fill=black,scale=0.15] (p10) {};
\path (1.2,-0.2) node[draw,shape=circle,fill=black,scale=0.15] (p11) {};
\path (1.2,-0.6) node[draw,shape=circle,fill=black,scale=0.15] (p12) {};
\path (1.8,0.6) node[draw,shape=circle,fill=black,scale=0.15] (p13) {};
\path (1.8,0.2) node[draw,shape=circle,fill=black,scale=0.15] (p14) {};
\path (1.8,-0.2) node[draw,shape=circle,fill=black,scale=0.15] (p15) {};
\path (1.8,-0.6) node[draw,shape=circle,fill=black,scale=0.15] (p16) {};
\path (2.2,0.6) node[draw,shape=circle,fill=black,scale=0.15] (p17) {};
\path (2.2,0.2) node[draw,shape=circle,fill=black,scale=0.15] (p18) {};
\path (2.2,-0.2) node[draw,shape=circle,fill=black,scale=0.15] (p19) {};
\path (2.2,-0.6) node[draw,shape=circle,fill=black,scale=0.15] (p20) {};
\path (2.8,0.6) node[draw,shape=circle,fill=black,scale=0.15] (p21) {};
\path (2.8,0.2) node[draw,shape=circle,fill=black,scale=0.15] (p22) {};
\path (2.8,-0.2) node[draw,shape=circle,fill=black,scale=0.15] (p23) {};
\path (2.8,-0.6) node[draw,shape=circle,fill=black,scale=0.15] (p24) {};
\draw  (p9) to (p10);
\draw  (p13) to (p14);
\draw  (p11) to (p12);
\draw  (p15) to (p16);
\draw  (p17) to (p18);
\draw  (p21) to (p22);
\draw  (p19) to (p20);
\draw  (p23) to (p24);
\end{tikzpicture}
\begin{tikzpicture}
 \path (0,1.05) node[draw=none] (p0) {\phantom{}};
 \path (0,0.0) node[anchor=north] (p1) {};
 \path (0,-1.05) node[draw=none] (p0) {\phantom{}};
\end{tikzpicture}$$

$$\begin{tikzpicture}
 \path (0,1.05) node[draw=none] (p0) {\phantom{$g_3=$}};
 \path (0,0.0) node[anchor=north] (p1) {$g_2=$};
 \path (0,-1.05) node[draw=none] (p0) {\phantom{$g_3=$}};
\end{tikzpicture}
\begin{tikzpicture}
\draw (0,0.9) -- (3.0,0.9) -- (3.0,-0.9) -- (0,-0.9) -- (0,0.9);
\draw (1.0,-0.9) -- (1.0,0.9);
\draw (2.0,-0.9) -- (2.0,0.9);
\path (0.2,0.6) node[draw,shape=circle,fill=black,scale=0.15] (p1) {};
\path (0.2,0.2) node[draw,shape=circle,fill=black,scale=0.15] (p2) {};
\path (0.2,-0.2) node[draw,shape=circle,fill=black,scale=0.15] (p3) {};
\path (0.2,-0.6) node[draw,shape=circle,fill=black,scale=0.15] (p4) {};
\path (0.8,0.6) node[draw,shape=circle,fill=black,scale=0.15] (p5) {};
\path (0.8,0.2) node[draw,shape=circle,fill=black,scale=0.15] (p6) {};
\path (0.8,-0.2) node[draw,shape=circle,fill=black,scale=0.15] (p7) {};
\path (0.8,-0.6) node[draw,shape=circle,fill=black,scale=0.15] (p8) {};
\path (1.2,0.6) node[draw,shape=circle,fill=black,scale=0.15] (p9) {};
\path (1.2,0.2) node[draw,shape=circle,fill=black,scale=0.15] (p10) {};
\path (1.2,-0.2) node[draw,shape=circle,fill=black,scale=0.15] (p11) {};
\path (1.2,-0.6) node[draw,shape=circle,fill=black,scale=0.15] (p12) {};
\path (1.8,0.6) node[draw,shape=circle,fill=black,scale=0.15] (p13) {};
\path (1.8,0.2) node[draw,shape=circle,fill=black,scale=0.15] (p14) {};
\path (1.8,-0.2) node[draw,shape=circle,fill=black,scale=0.15] (p15) {};
\path (1.8,-0.6) node[draw,shape=circle,fill=black,scale=0.15] (p16) {};
\path (2.2,0.6) node[draw,shape=circle,fill=black,scale=0.15] (p17) {};
\path (2.2,0.2) node[draw,shape=circle,fill=black,scale=0.15] (p18) {};
\path (2.2,-0.2) node[draw,shape=circle,fill=black,scale=0.15] (p19) {};
\path (2.2,-0.6) node[draw,shape=circle,fill=black,scale=0.15] (p20) {};
\path (2.8,0.6) node[draw,shape=circle,fill=black,scale=0.15] (p21) {};
\path (2.8,0.2) node[draw,shape=circle,fill=black,scale=0.15] (p22) {};
\path (2.8,-0.2) node[draw,shape=circle,fill=black,scale=0.15] (p23) {};
\path (2.8,-0.6) node[draw,shape=circle,fill=black,scale=0.15] (p24) {};
\draw  (p5) to [bend left=45] (p9);
\draw  (p6) to [bend right=45] (p10);
\draw  (p11) to [bend right=45] (p7);
\draw  (p8) to [bend right=45] (p12);
\draw  (p1) to  [bend right=15](p13);
\draw  (p2) to [bend left=15] (p14);
\draw  (p3) to [bend right=15] (p15);
\draw  (p4) to [bend left=15] (p16);
\end{tikzpicture}
\begin{tikzpicture}
 \path (0,1.05) node[draw=none] (p0) {\phantom{}};
 \path (0,0.0) node[anchor=north] (p1) {};
 \path (0,-1.05) node[draw=none] (p0) {\phantom{}};
\end{tikzpicture}$$

$$\begin{tikzpicture}
 \path (0,1.05) node[draw=none] (p0) {\phantom{$g_3=$}};
 \path (0,0.0) node[anchor=north] (p1) {$g_3=$};
 \path (0,-1.05) node[draw=none] (p0) {\phantom{$g_3=$}};
\end{tikzpicture}
\begin{tikzpicture}
\draw (0,0.9) -- (3.0,0.9) -- (3.0,-0.9) -- (0,-0.9) -- (0,0.9);
\draw (1.0,-0.9) -- (1.0,0.9);
\draw (2.0,-0.9) -- (2.0,0.9);
\path (0.2,0.6) node[draw,shape=circle,fill=black,scale=0.15] (p1) {};
\path (0.2,0.2) node[draw,shape=circle,fill=black,scale=0.15] (p2) {};
\path (0.2,-0.2) node[draw,shape=circle,fill=black,scale=0.15] (p3) {};
\path (0.2,-0.6) node[draw,shape=circle,fill=black,scale=0.15] (p4) {};
\path (0.8,0.6) node[draw,shape=circle,fill=black,scale=0.15] (p5) {};
\path (0.8,0.2) node[draw,shape=circle,fill=black,scale=0.15] (p6) {};
\path (0.8,-0.2) node[draw,shape=circle,fill=black,scale=0.15] (p7) {};
\path (0.8,-0.6) node[draw,shape=circle,fill=black,scale=0.15] (p8) {};
\path (1.2,0.6) node[draw,shape=circle,fill=black,scale=0.15] (p9) {};
\path (1.2,0.2) node[draw,shape=circle,fill=black,scale=0.15] (p10) {};
\path (1.2,-0.2) node[draw,shape=circle,fill=black,scale=0.15] (p11) {};
\path (1.2,-0.6) node[draw,shape=circle,fill=black,scale=0.15] (p12) {};
\path (1.8,0.6) node[draw,shape=circle,fill=black,scale=0.15] (p13) {};
\path (1.8,0.2) node[draw,shape=circle,fill=black,scale=0.15] (p14) {};
\path (1.8,-0.2) node[draw,shape=circle,fill=black,scale=0.15] (p15) {};
\path (1.8,-0.6) node[draw,shape=circle,fill=black,scale=0.15] (p16) {};
\path (2.2,0.6) node[draw,shape=circle,fill=black,scale=0.15] (p17) {};
\path (2.2,0.2) node[draw,shape=circle,fill=black,scale=0.15] (p18) {};
\path (2.2,-0.2) node[draw,shape=circle,fill=black,scale=0.15] (p19) {};
\path (2.2,-0.6) node[draw,shape=circle,fill=black,scale=0.15] (p20) {};
\path (2.8,0.6) node[draw,shape=circle,fill=black,scale=0.15] (p21) {};
\path (2.8,0.2) node[draw,shape=circle,fill=black,scale=0.15] (p22) {};
\path (2.8,-0.2) node[draw,shape=circle,fill=black,scale=0.15] (p23) {};
\path (2.8,-0.6) node[draw,shape=circle,fill=black,scale=0.15] (p24) {};
\draw  (p1) to [bend left=45] (p3);
\draw  (p2) to  (p7);
\draw  (p4) to (p6);
\draw  (p5) to [bend right=45] (p8);
\draw  (p9) to  (p10);
\draw  (p11) to [bend right=15] (p23);
\draw  (p12) to [bend left=15] (p24);
\draw  (p13) to (p14);
\draw  (p15) to [bend right=15] (p19);
\draw  (p16) to [bend left=15] (p20);
\draw  (p17) to  (p18);
\draw  (p21) to  (p22);
\end{tikzpicture}
\begin{tikzpicture}
 \path (0,1.05) node[draw=none] (p0) {\phantom{}};
 \path (0,0.0) node[anchor=north] (p1) {};
 \path (0,-1.05) node[draw=none] (p0) {\phantom{}};
\end{tikzpicture}$$
\end{multicols}
\begin{multicols}{3}
$$\begin{tikzpicture}
 \path (0,1.05) node[draw=none] (p0) {\phantom{$g_3=$}};
 \path (0,0.0) node[anchor=north] (p1) {$g_4=$};
 \path (0,-1.05) node[draw=none] (p0) {\phantom{$g_3=$}};
\end{tikzpicture}
\begin{tikzpicture}
\draw (0,0.9) -- (3.0,0.9) -- (3.0,-0.9) -- (0,-0.9) -- (0,0.9);
\draw (1.0,-0.9) -- (1.0,0.9);
\draw (2.0,-0.9) -- (2.0,0.9);
\path (0.2,0.6) node[draw,shape=circle,fill=black,scale=0.15] (p1) {};
\path (0.2,0.2) node[draw,shape=circle,fill=black,scale=0.15] (p2) {};
\path (0.2,-0.2) node[draw,shape=circle,fill=black,scale=0.15] (p3) {};
\path (0.2,-0.6) node[draw,shape=circle,fill=black,scale=0.15] (p4) {};
\path (0.8,0.6) node[draw,shape=circle,fill=black,scale=0.15] (p5) {};
\path (0.8,0.2) node[draw,shape=circle,fill=black,scale=0.15] (p6) {};
\path (0.8,-0.2) node[draw,shape=circle,fill=black,scale=0.15] (p7) {};
\path (0.8,-0.6) node[draw,shape=circle,fill=black,scale=0.15] (p8) {};
\path (1.2,0.6) node[draw,shape=circle,fill=black,scale=0.15] (p9) {};
\path (1.2,0.2) node[draw,shape=circle,fill=black,scale=0.15] (p10) {};
\path (1.2,-0.2) node[draw,shape=circle,fill=black,scale=0.15] (p11) {};
\path (1.2,-0.6) node[draw,shape=circle,fill=black,scale=0.15] (p12) {};
\path (1.8,0.6) node[draw,shape=circle,fill=black,scale=0.15] (p13) {};
\path (1.8,0.2) node[draw,shape=circle,fill=black,scale=0.15] (p14) {};
\path (1.8,-0.2) node[draw,shape=circle,fill=black,scale=0.15] (p15) {};
\path (1.8,-0.6) node[draw,shape=circle,fill=black,scale=0.15] (p16) {};
\path (2.2,0.6) node[draw,shape=circle,fill=black,scale=0.15] (p17) {};
\path (2.2,0.2) node[draw,shape=circle,fill=black,scale=0.15] (p18) {};
\path (2.2,-0.2) node[draw,shape=circle,fill=black,scale=0.15] (p19) {};
\path (2.2,-0.6) node[draw,shape=circle,fill=black,scale=0.15] (p20) {};
\path (2.8,0.6) node[draw,shape=circle,fill=black,scale=0.15] (p21) {};
\path (2.8,0.2) node[draw,shape=circle,fill=black,scale=0.15] (p22) {};
\path (2.8,-0.2) node[draw,shape=circle,fill=black,scale=0.15] (p23) {};
\path (2.8,-0.6) node[draw,shape=circle,fill=black,scale=0.15] (p24) {};
\draw  (p1) to  (p2);
\draw  (p3) to  (p4);
\draw  (p5) to [bend right=30] (p8);
\draw  (p6) to  [bend left=30] (p7);
\draw  (p9) to  [bend left=30]  (p12);
\draw  (p10) to [bend right=30] (p11);
\draw  (p13) to  (p14);
\draw  (p15) to (p16);
\draw  (p17) to [bend left=30] (p20);
\draw  (p18) to [bend right=30] (p19);
\draw  (p21) to  (p22);
\draw  (p23) to  (p24);
\end{tikzpicture}
\begin{tikzpicture}
 \path (0,1.05) node[draw=none] (p0) {\phantom{}};
 \path (0,0.0) node[anchor=north] (p1) {};
 \path (0,-1.05) node[draw=none] (p0) {\phantom{}};
\end{tikzpicture}$$

$$\begin{tikzpicture}
 \path (0,1.05) node[draw=none] (p0) {\phantom{$g_3=$}};
 \path (0,0.0) node[anchor=north] (p1) {$g_5=$};
 \path (0,-1.05) node[draw=none] (p0) {\phantom{$g_3=$}};
\end{tikzpicture}
\begin{tikzpicture}
\draw (0,0.9) -- (3.0,0.9) -- (3.0,-0.9) -- (0,-0.9) -- (0,0.9);
\draw (1.0,-0.9) -- (1.0,0.9);
\draw (2.0,-0.9) -- (2.0,0.9);
\path (0.2,0.6) node[draw,shape=circle,fill=black,scale=0.15] (p1) {};
\path (0.2,0.2) node[draw,shape=circle,fill=black,scale=0.15] (p2) {};
\path (0.2,-0.2) node[draw,shape=circle,fill=black,scale=0.15] (p3) {};
\path (0.2,-0.6) node[draw,shape=circle,fill=black,scale=0.15] (p4) {};
\path (0.8,0.6) node[draw,shape=circle,fill=black,scale=0.15] (p5) {};
\path (0.8,0.2) node[draw,shape=circle,fill=black,scale=0.15] (p6) {};
\path (0.8,-0.2) node[draw,shape=circle,fill=black,scale=0.15] (p7) {};
\path (0.8,-0.6) node[draw,shape=circle,fill=black,scale=0.15] (p8) {};
\path (1.2,0.6) node[draw,shape=circle,fill=black,scale=0.15] (p9) {};
\path (1.2,0.2) node[draw,shape=circle,fill=black,scale=0.15] (p10) {};
\path (1.2,-0.2) node[draw,shape=circle,fill=black,scale=0.15] (p11) {};
\path (1.2,-0.6) node[draw,shape=circle,fill=black,scale=0.15] (p12) {};
\path (1.8,0.6) node[draw,shape=circle,fill=black,scale=0.15] (p13) {};
\path (1.8,0.2) node[draw,shape=circle,fill=black,scale=0.15] (p14) {};
\path (1.8,-0.2) node[draw,shape=circle,fill=black,scale=0.15] (p15) {};
\path (1.8,-0.6) node[draw,shape=circle,fill=black,scale=0.15] (p16) {};
\path (2.2,0.6) node[draw,shape=circle,fill=black,scale=0.15] (p17) {};
\path (2.2,0.2) node[draw,shape=circle,fill=black,scale=0.15] (p18) {};
\path (2.2,-0.2) node[draw,shape=circle,fill=black,scale=0.15] (p19) {};
\path (2.2,-0.6) node[draw,shape=circle,fill=black,scale=0.15] (p20) {};
\path (2.8,0.6) node[draw,shape=circle,fill=black,scale=0.15] (p21) {};
\path (2.8,0.2) node[draw,shape=circle,fill=black,scale=0.15] (p22) {};
\path (2.8,-0.2) node[draw,shape=circle,fill=black,scale=0.15] (p23) {};
\path (2.8,-0.6) node[draw,shape=circle,fill=black,scale=0.15] (p24) {};
\draw  (p1) to  (p5);
\draw  (p2) to  (p6);
\draw  (p3) to  (p8);
\draw  (p4) to (p7);
\draw  (p9) to   (p13);
\draw  (p10) to  (p14);
\draw  (p11) to  (p16);
\draw  (p12) to (p15);
\draw  (p17) to (p21);
\draw  (p18) to  (p22);
\draw  (p19) to  (p24);
\draw  (p20) to  (p23);
\end{tikzpicture}
\begin{tikzpicture}
 \path (0,1.05) node[draw=none] (p0) {\phantom{}};
 \path (0,0.0) node[anchor=north] (p1) {};
 \path (0,-1.05) node[draw=none] (p0) {\phantom{}};
\end{tikzpicture}$$

$$\begin{tikzpicture}
 \path (0,1.05) node[draw=none] (p0) {\phantom{$g_3=$}};
 \path (0,0.0) node[anchor=north] (p1) {$g_6=$};
 \path (0,-1.05) node[draw=none] (p0) {\phantom{$g_3=$}};
\end{tikzpicture}
\begin{tikzpicture}
\draw (0,0.9) -- (3.0,0.9) -- (3.0,-0.9) -- (0,-0.9) -- (0,0.9);
\draw (1.0,-0.9) -- (1.0,0.9);
\draw (2.0,-0.9) -- (2.0,0.9);
\path (0.2,0.6) node[draw,shape=circle,fill=black,scale=0.15] (p1) {};
\path (0.2,0.2) node[draw,shape=circle,fill=black,scale=0.15] (p2) {};
\path (0.2,-0.2) node[draw,shape=circle,fill=black,scale=0.15] (p3) {};
\path (0.2,-0.6) node[draw,shape=circle,fill=black,scale=0.15] (p4) {};
\path (0.8,0.6) node[draw,shape=circle,fill=black,scale=0.15] (p5) {};
\path (0.8,0.2) node[draw,shape=circle,fill=black,scale=0.15] (p6) {};
\path (0.8,-0.2) node[draw,shape=circle,fill=black,scale=0.15] (p7) {};
\path (0.8,-0.6) node[draw,shape=circle,fill=black,scale=0.15] (p8) {};
\path (1.2,0.6) node[draw,shape=circle,fill=black,scale=0.15] (p9) {};
\path (1.2,0.2) node[draw,shape=circle,fill=black,scale=0.15] (p10) {};
\path (1.2,-0.2) node[draw,shape=circle,fill=black,scale=0.15] (p11) {};
\path (1.2,-0.6) node[draw,shape=circle,fill=black,scale=0.15] (p12) {};
\path (1.8,0.6) node[draw,shape=circle,fill=black,scale=0.15] (p13) {};
\path (1.8,0.2) node[draw,shape=circle,fill=black,scale=0.15] (p14) {};
\path (1.8,-0.2) node[draw,shape=circle,fill=black,scale=0.15] (p15) {};
\path (1.8,-0.6) node[draw,shape=circle,fill=black,scale=0.15] (p16) {};
\path (2.2,0.6) node[draw,shape=circle,fill=black,scale=0.15] (p17) {};
\path (2.2,0.2) node[draw,shape=circle,fill=black,scale=0.15] (p18) {};
\path (2.2,-0.2) node[draw,shape=circle,fill=black,scale=0.15] (p19) {};
\path (2.2,-0.6) node[draw,shape=circle,fill=black,scale=0.15] (p20) {};
\path (2.8,0.6) node[draw,shape=circle,fill=black,scale=0.15] (p21) {};
\path (2.8,0.2) node[draw,shape=circle,fill=black,scale=0.15] (p22) {};
\path (2.8,-0.2) node[draw,shape=circle,fill=black,scale=0.15] (p23) {};
\path (2.8,-0.6) node[draw,shape=circle,fill=black,scale=0.15] (p24) {};
\draw  (p1) to  (p6);
\draw  (p2) to  (p5);
\draw  (p3) to  (p4);
\draw  (p7) to (p8);
\draw  (p9) to   (p14);
\draw  (p10) to  (p13);
\draw  (p11) to  (p12);
\draw  (p15) to (p16);
\draw  (p17) to (p22);
\draw  (p18) to  (p21);
\draw  (p19) to  (p20);
\draw  (p23) to  (p24);
\end{tikzpicture}
\begin{tikzpicture}
 \path (0,1.05) node[draw=none] (p0) {\phantom{}};
 \path (0,0.0) node[anchor=north] (p1) {};
 \path (0,-1.05) node[draw=none] (p0) {\phantom{}};
\end{tikzpicture}$$
\end{multicols}

Three duads (that is, 2-element subsets) of $\Omega$ which will appear in our arguments are

$$\Delta =  \begin{array}{c}
\smalloctada{}{}{}{}{}{}{}{}
\smalloctadb{}{}{}{}{}{}{}{}
\smalloctadc{}{}{}{}{\times}{\times}{}{}{},
\end{array}
\text{ } 
\Delta_1 =  \begin{array}{c}
\smalloctada{}{}{}{}{}{}{}{}
\smalloctadb{}{}{}{}{}{}{}{}
\smalloctadc{\times}{\times}{}{}{}{}{}{}{}
\end{array}  \text{ and } 
\Delta_2 =  \begin{array}{c}
\smalloctada{}{}{}{}{}{}{}{}
\smalloctadb{}{}{}{}{\times}{\times}{}{}
\smalloctadc{}{}{}{}{}{}{}{}{}.
\end{array}$$

For $\{j_1, \dots, j_r \} \subseteq \{1, \dots, 6 \}$, we put $G_{j_1 \dots j_r} = \langle g_{j_1}, \dots, g_{j_r} \rangle$.  Setting $t_i = g_i$ for $i = 1, \dots ,5$ and $s_i = g_i$ for $i = 1, \dots, 4$ and $s_5 = g_6$, we have the following result.

\begin{lemma}\label{L2.1}
\begin{enumerate}

\item[(i)] $\langle g_1, g_2, g_3, g_4 \rangle  
= \Stab_G(\Delta) \cong \M_{22}:2, \langle g_2, g_3, g_4 \rangle \cong \PSL_2(11):2$ and 
$\langle g_2, g_3, g_4, g_5 \rangle = \langle g_2, g_3, g_4, g_6 \rangle  
 = \Stab_G(\{ D, \Omega \setminus {D} \}) \cong \M_{12}:2 $.
 \item[(ii)] $\{t_1, t_2, t_3, t_4, t_5 \}$ is a C-string for $G$ with Schlafli symbol
$$ 
 \begin{picture}(220,80)(0,-30)
    \put(20,20){\circle{4}}
    \put(62,20){\line(1,0){36}}
    \put(60,20){\circle{4}}
    \put(22,20){\line(1,0){36}}
	\put(102,20){\line(1,0){36}}
     \put(142,20){\line(1,0){36}}
    \put(100,20){\circle{4}}
    \put(140,20){\circle{4}}
    \put(180,20){\circle{4}}
    \put(17,0){$t_1$}
    \put(57,0){$t_2$}
    \put(97,0){$t_3$}
    \put(137,0){$t_4$}
    \put(177,0){$t_5$}
     \put(38,25){$4$}
	\put(72,25){$10$}
	\put(117,25){$3$}
	\put(157,25){$4$}
     \put(190,18){.}
\end{picture}$$ 
\item[(iii)] $\{s_1, s_2, s_3, s_4, s_5 \}$ is a C-string for $G$ with Schlafli symbol
$$ 
 \begin{picture}(220,80)(0,-30)
    \put(20,20){\circle{4}}
    \put(62,20){\line(1,0){36}}
    \put(60,20){\circle{4}}
    \put(22,20){\line(1,0){36}}
	\put(102,20){\line(1,0){36}}
     \put(142,20){\line(1,0){36}}
    \put(100,20){\circle{4}}
    \put(140,20){\circle{4}}
    \put(180,20){\circle{4}}
    \put(17,0){$s_1$}
    \put(57,0){$s_2$}
    \put(97,0){$s_3$}
    \put(137,0){$s_4$}
    \put(177,0){$s_5$}
     \put(38,25){$4$}
	\put(72,25){$10$}
	\put(117,25){$3$}
	\put(157,25){$3$}
     \put(190,18){.}
\end{picture}$$ 

\end{enumerate}
\end{lemma}

\begin{proof}
The orders of $g_ig_j$ as indicated in parts (ii) and (iii) may be readily verified. \\
(i) We observe that both $G_{234}$ and $G_{1234}$ have two orbits on $\Omega$, namely $\Delta$ and $\Omega \setminus \Delta$. Therefore $G_{1234} \leq K = \Stab_G(\Delta) \cong \M_{22}:2$. Also we note that both $G_{2345}$ and $G_{2346}$ leave the duum $\{D, \Omega \setminus D \}$ invariant and therefore $G_{2345}, G_{2346} \leq  L = \Stab_G(\{D, \Omega \setminus D \}) \cong M_{12}:2$. Since $D^{g_1}$ is not equal to $D$ or $\Omega \setminus D$, $g_1 \notin L$. In particular, $g_1 \notin G_{234}$ and so $G_{234} \ne G_{1234}$. By Table 3 of \cite{conway} we have $G_{1234}$ is isomorphic to either $\M_{22}:2$ or $\PSL_2(11):2$. If the latter holds, then, as $G_{234} < G_{1234}$, and $|G_{234}|$ is divisible by 3 and 5, we must have $G_{234} \cong \PSL_2(11)$. But then $\Delta$ cannot be a $G_{234}$ orbit. Hence $G_{1234} \cong \M_{22}:2$, and we then also conclude that $G_{234} \cong \PSL_2(11):2$. Now $G_{2345}$ and $G_{2346}$ are each transitive on $\Omega$ and, because $\Delta$ is not invariant under either of $g_5$ and $g_6$, $G_{234}  < G_{2345}$ and $G_{234}  < G_{2346}$. Thus the only possibility is that $G_{2345} = L = G_{2346}$, and we have part (i).\\
We now prove parts (ii) and (iii), making repeated use of [2E16(a); \cite{mcmullenschulte}]. First we show that $\{g_1, g_2, g_3, g_4 \}$ is a C-string for $G_{1234}$. Looking at $G_{123}$, if $G_{12} \cap G_{23} \ne G_2$, then, as $G_{12} \cong \Dih(8)$ and $G_{12} \cong \Dih(20)$, we must have $(g_2g_3)^5 \in G_{12}$. But $(g_2g_3)^5$ interchanges 5 and 12, which are in different $G_{12}$-orbits. Therefore $G_{12} \cap G_{23} = G_2$. So, by [2E16(a); \cite{mcmullenschulte}], $\{g_1, g_2, g_3 \}$ is a C-string for $G_{123}$. Since $g_4 \notin G_{23}$ (as, for example, $g_4$ interchanges 21 and 22), $\{g_2, g_3, g_4 \}$ is a C-string for $G_{234}$. By part (i) $G_{234} \cong \PSL_2(11):2$ and so, by \cite{atlas}, $G_{23} \cong \Dih(20)$ is a maximal subgroup of $G_{234}$. Hence $G_{123} \cap G_{234} > G_{23}$ would force $G_{123} \leq G_{234}$ which is impossible. Therefore $G_{123} \cap G_{234} =  G_{23}$ whence, by  [2E16(a); \cite{mcmullenschulte}], $\{g_1, g_2, g_3, g_4 \}$ is a C-string for $G_{1234}$.\\
Our attention now moves to $G_{2345} = G_{2346}$. We already know that $\{g_2, g_3, g_4 \}$ is a C-string for $G_{234}$. Now the strings for $\{g_3, g_4, g_5\}$ and $\{g_3, g_4, g_6\}$ are Coxeter diagrams and we see that $G_{345} \cong B_3$ and $G_{346} \cong \Sym(4)$. We note that $g_6 \notin G_{234}$ and so clearly $G_{346} \cap G_{234} = G_{34}$. We also observe that $g_3g_4g_3g_5g_4g_3g_5g_4g_5  \notin  G_{234}$ (as it maps 19 to 11) and $(g_4g_5)^2 \notin G_{234}$ (as it maps 19 to 6). Hence, from the structure of $B_3$, we also have $G_{345} \cap G_{234} = G_{34}$. Thus $\{g_2, g_3, g_4, g_5 \}$ and $\{g_2, g_3, g_4, g_6 \}$ are C-strings for $G_{2345}=G_{2346}$. \\
Finally we consider $G_{1234} \cap G_{2345}$ and $G_{1234} \cap G_{2346}$. If they are not equal to $G_{234} \cong \PSL_2(11):2$, which is a maximal subgroup of $G_{1234}$, then $G_{1234} \leq G_{2345}$ or $G_{1234} \leq G_{2346}$, neither of which is possible. Calling upon [2E16(a); \cite{mcmullenschulte}] yet again yields parts (i) and (ii).
\end{proof}

\begin{remark} The content of part (i) of Lemma \ref{L2.1} was also observed in \cite{hartleyhulpke}.

\end{remark}

\begin{lemma}\label{L2.3} The chamber graph of the C-string $\{g_2, g_3, g_4 \}$ has diameter 24 and disc sizes as follows.
$$\begin{tabular}{c|c|c|c|c|c|c|c|c|c|c|c|c}
$i$ & 1 & 2 & 3 & 4& 5 & 6& 7& 8& 9&10&11&12\\ \hline
$|\Delta_i(1)|$ & 3 & 5 & 7 & 9& 12 & 16& 21& 28& 37&48&61&77 \\
\end{tabular}$$
$$\begin{tabular}{c|c|c|c|c|c|c|c|c|c|c|c|c}
$i$ & 13 & 14 & 15 & 16& 17 & 18& 19& 20& 21&22&23&24\\ \hline
$|\Delta_i(1)|$ & 98 & 126 & 162 & 163& 138 & 110& 95& 60& 26&11&5&1 \\
\end{tabular}$$
\end{lemma}

\begin{proof} This is a quick calculation using \textsc{Magma} \cite{magma}.
\end{proof}

\begin{remark} The only information from Lemma \ref{L2.3} we need in the proof of Theorem~\ref{maintheorem} is that the diameter is 24 -- the disc sizes are recorded as they may be of interest. Also we note that the chamber at maximal distance from the chamber corresponding to $1$ corresponds to the group element

$$
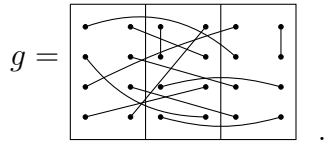
\begin{tikzpicture}
 \path (0,1.05) node[draw=none] (p0) {\phantom{$g=$}};
 \path (0,0.0) node[anchor=north] (p1) {$g=$};
 \path (0,-1.05) node[draw=none] (p0) {\phantom{$g=$}};
\end{tikzpicture}
\begin{tikzpicture}
\draw (0,0.9) -- (3.0,0.9) -- (3.0,-0.9) -- (0,-0.9) -- (0,0.9);
\draw (1.0,-0.9) -- (1.0,0.9);
\draw (2.0,-0.9) -- (2.0,0.9);
\path (0.2,0.6) node[draw,shape=circle,fill=black,scale=0.15] (p1) {};
\path (0.2,0.2) node[draw,shape=circle,fill=black,scale=0.15] (p2) {};
\path (0.2,-0.2) node[draw,shape=circle,fill=black,scale=0.15] (p3) {};
\path (0.2,-0.6) node[draw,shape=circle,fill=black,scale=0.15] (p4) {};
\path (0.8,0.6) node[draw,shape=circle,fill=black,scale=0.15] (p5) {};
\path (0.8,0.2) node[draw,shape=circle,fill=black,scale=0.15] (p6) {};
\path (0.8,-0.2) node[draw,shape=circle,fill=black,scale=0.15] (p7) {};
\path (0.8,-0.6) node[draw,shape=circle,fill=black,scale=0.15] (p8) {};
\path (1.2,0.6) node[draw,shape=circle,fill=black,scale=0.15] (p9) {};
\path (1.2,0.2) node[draw,shape=circle,fill=black,scale=0.15] (p10) {};
\path (1.2,-0.2) node[draw,shape=circle,fill=black,scale=0.15] (p11) {};
\path (1.2,-0.6) node[draw,shape=circle,fill=black,scale=0.15] (p12) {};
\path (1.8,0.6) node[draw,shape=circle,fill=black,scale=0.15] (p13) {};
\path (1.8,0.2) node[draw,shape=circle,fill=black,scale=0.15] (p14) {};
\path (1.8,-0.2) node[draw,shape=circle,fill=black,scale=0.15] (p15) {};
\path (1.8,-0.6) node[draw,shape=circle,fill=black,scale=0.15] (p16) {};
\path (2.2,0.6) node[draw,shape=circle,fill=black,scale=0.15] (p17) {};
\path (2.2,0.2) node[draw,shape=circle,fill=black,scale=0.15] (p18) {};
\path (2.2,-0.2) node[draw,shape=circle,fill=black,scale=0.15] (p19) {};
\path (2.2,-0.6) node[draw,shape=circle,fill=black,scale=0.15] (p20) {};
\path (2.8,0.6) node[draw,shape=circle,fill=black,scale=0.15] (p21) {};
\path (2.8,0.2) node[draw,shape=circle,fill=black,scale=0.15] (p22) {};
\path (2.8,-0.2) node[draw,shape=circle,fill=black,scale=0.15] (p23) {};
\path (2.8,-0.6) node[draw,shape=circle,fill=black,scale=0.15] (p24) {};
\draw  (p1) to [bend left=30]  (p18);
\draw  (p2) to [bend right=25]  (p16);
\draw  (p3) to [bend left=5]  (p17);
\draw  (p4) to (p15);
\draw  (p5) to (p14);
\draw  (p6) to (p19);
\draw  (p7) to (p20);
\draw  (p8) to (p13);
\draw  (p9) to (p10);
\draw  (p11) to [bend left=15] (p23);
\draw  (p12) to [bend right=15] (p24);
\draw  (p21) to (p22);
\end{tikzpicture}
\begin{tikzpicture}
 \path (0,1.05) node[draw=none] (p0) {\phantom{}};
 \path (0,0.0) node[anchor=north] (p1) {};
 \path (0,-1.05) node[draw=none] (p0) {\phantom{}};
\end{tikzpicture}.$$

\end{remark}

We require one further result on $G_{234}$. Put $\Lambda = \Omega \setminus \Delta$ and $E = D \cap \Lambda$. Then $|E| = 11$ and $G_{234}$ preserves the partition $\{ E, \Lambda \setminus E \}$. Set 
$$\Theta = \{ \{ \alpha, \beta \} \; | \; \alpha \in E, \beta \in  \Lambda \setminus E \}.$$
Observe that $G_{234}$ acts upon $\Theta$ and that $\Delta_1, \Delta_2 \in \Theta$.

\begin{lemma}\label{different} The duads $\Delta_1$ and $\Delta_2$ are in different $G_{234}$-orbits.
\end{lemma}

\begin{proof} Clearly $\langle g_2,g_3 \rangle \leq Stab_{G_{234}}(\Delta_1)$. If $\langle g_2,g_3 \rangle \neq Stab_{G_{234}}(\Delta_1)$, then, by \textsc{Atlas}\cite{atlas}, $Stab_{G_{234}}(\Delta_1) = G_{234} $, which is not the case. Hence 
$$ \Dih(20) \cong \langle g_2,g_3 \rangle  = Stab_{G_{234}}(\Delta_1).$$
Since $ \Dih(6) \cong \langle g_3,g_4 \rangle  \leq Stab_{G_{234}}(\Delta_2)$, it follows that $\Delta_1$ and $\Delta_2$ are in different $G_{234}$-orbits.
\end{proof}

\textbf{Proof of Theorem 1.1}

 From Lemma \ref{L2.1}(i) $G_{234} \cong \PSL_2(11):2$ and $G_{2345} =  G_{2346} \cong \M_{12}:2$. Consulting the \textsc{Atlas}\cite{atlas}, gives the permutation characters for $G_{234}$ in $G_{2345} =G_{2346}$ and for $G_{2345}$ in $G$. From this we determine that $G_{2345} = G_{2346}$ acting on the right cosets of $G_{234}$ has rank 4 and $G$ acting on the right cosets of $G_{2345}$ has rank 3.\\

(1)(i) Double coset representatives for $G_{234}$ in $G_{2345}$ are $1, g_5, g_5g_4g_5g_3g_4g_5$ and\\
 $g_5g_4g_3g_2g_3g_2g_4g_3g_5g_4g_5$.\\
(ii) Double coset representatives for $G_{234}$ in $G_{2346}$ are $1, g_6,  g_6g_4g_3g_2g_3g_2g_3g_2g_3g_2g_4g_6$ and\\
$g_6g_4g_3g_2g_3g_2g_4g_3g_2g_3g_2g_3g_4g_6.$\\

 Recall that $\Delta$ is a $G_{234}$-orbit of $\Omega$. Setting 
$x_1 =   g_5g_4g_5g_3g_4g_5$ and $x_2 =
g_5g_4g_3g_2g_3g_2g_4g_3g_5g_4g_5$, we see that $\Delta^{g_5} = \Delta_1$, $\Delta^{x_1} = \Delta_2$ and $|\Delta \cap \Delta^{x_2}| = 1$. 
If $g_5$ and $x_1$ are in the same $G_{234}$-double coset, then $x_1 = h_1g_5h_2$ for some $h_1, h_2 \in G_{234}$. Then
$$\Delta_2 = \Delta^{x_1} = \Delta^{h_1g_5h_2} = \Delta^{g_5h_2} = \Delta_1^{h_2},$$
which contradicts Lemma~\ref{different}. Hence $g_5$ and $x_1$ are in different double cosets of $G_{234}$. Since $\Delta^1 = \Delta$ and $|\Delta \cap \Delta^{x_2}| = 1$, we see that $1$, $g_5, x_1$ and $x_2$ are in different $G_{234}$-double cosets of $G_{2345}$. Thus, as the rank of $G_{2345}$ on the right cosets of $G_{234}$ is 4, (1)(i) follows.\\

Turning to part (ii), this time we set $y_1 = g_6g_4g_3g_2g_3g_2g_3g_2g_3g_2g_4g_6$ and \\ $y_2 = g_6g_4g_3g_2g_3g_2g_4g_3g_2g_3g_2g_3g_4g_6$. We check that $\Delta^{g_6} = \Delta_1$, $\Delta^{y_1} = \Delta_2$ and $|\Delta \cap \Delta^{y_2}| = 1$. Again using Lemma~\ref{different} we may argue as in (i) to deduce that $1$, $g_6, y_1$ and $y_2$ are in different double $G_{234}$-cosets of $G_{2346}$, whence we obtain (1)(ii).\\

(2) Double coset representatives for $G_{2345} = G_{2346}$ in $G$ are $1, g_1$ and $g_1g_2g_3g_2g_3g_4g_3g_2g_1.$\\

Put $h = g_1g_2g_3g_2g_3g_4g_3g_2g_1$. Then we see that $D \cap D^{g_1} = \{0,3,8,18\}$ and $D \cap D^h =\{1,4,5,8,9,10 \}$. Consequently $g_1$ and $h$ are in different double $G_{2345}$ cosets of $G$. Hence, as the rank of $G$ on the right cosets of $G_{2345} = G_{2346}$ is 3, we have (2).\\

Now (1)(i) and Lemma \ref{L2.3} imply that a word in $\{t_2, t_3, t_4, t_5\}$ has length at most $24 + 11  + 24 =59$. Then, employing (2), a word in $\{t_1, t_2, t_3, t_4, t_5\}$ has length at most $ 59 + 9 + 59 = 127$. This establishes Theorem \ref{maintheorem} for the C-string $\{t_1, t_2, t_3, t_4, t_5\}$ . For the C-string $\{s_1, s_2, s_3, s_4, s_5 \}$ a similar argument using (1)(ii) shows that its chamber graph has diameter at most 133, completing the proof of Theorem \ref{maintheorem}.\\

\end{document}